\documentclass[11pt,a4paper]{article}

\usepackage{inputenc}
\usepackage{amsmath}
\usepackage{bm}
\usepackage{bbold}
\usepackage{amsthm}

\usepackage{hyperref}

\title{An Analysis of the Least Median of Squares Regression Problem\thanks{Computational Statistics Vol. 1: COMPSTAT Proceedings of the 10th Symposium on Computational Statistics, Neuchatel, Switzerland, August 1992. Physica-Verlag, Heidelberg, 1992, pp~471-476.}
}

\author{Nikolai Krivulin\thanks{Faculty of Mathematics and Mechanics, St.~Petersburg State University}}

\date{}

\newtheorem{theorem}{Theorem}
\newtheorem{lemma}{Lemma}

\begin{document}

\maketitle

\begin{abstract}
The optimization problem that arises out of the least median of squared residuals method in linear regression is analyzed. To simplify the analysis, the problem is replaced by an equivalent one of minimizing the median of absolute residuals. A useful representation of the last problem is given to examine properties of the objective function and estimate the number of its local minima. It is shown that the exact number of local minima is equal to $  {p+\lfloor (n-1)/2 \rfloor \choose{p}} $, where $  p  $ is the dimension of the regression model and $  n  $ is the number of observations. As applications of the results, three algorithms are also outlined.
\end{abstract}

\section{Introduction}

The {\it least median of squares\/} (LMS) method has recently been proposed by
Rousseeuw in \cite{Rous84} to provide a very robust estimate of parameters in
linear regression problems. The LMS estimate can be obtained as the solution
of the following optimization problem.

Let $  {x}^{\top}_{i} = (x_{i1}, \ldots ,x_{ip}),\: i=1, \ldots ,n $, and
$  {y}=(y_{1}, \ldots ,y_{n})^{\top}  $ be given real vectors. We assume
that $  n/2 \geq p  $ and the $  (n \times p) $--matrix
$  X = [x_{ij}]  $ is of full rank to avoid degenerate cases. Let
$  {\theta} = (\theta_{1}, \ldots , \theta_{p})^{\top}  $ be a vector of
regression parameters. The optimization problem that arises out of the LMS
method is to find $  \theta^{*}  $ providing
\begin{equation} \label{equ1}
\min_{\theta} \; \mbox{med} \:
\{ (y_{i} - {x}_{i}^{\top} {\theta})^{2} \}.
\end{equation}

It is known (see \cite{Edel90,Joss90,Rous84,Souv87}) that the objective
function in (\ref{equ1}) is hard to minimize. This function is multi-extremal,
it is considered as having $  O(n^{p})  $ local minima. In fact, there
are efficient and exact algorithms available only for the problems of the
lowest dimensions. A simple algorithm for $  p=1  $ can be found in
\cite{Rous84}. Another two designed for the problems of the dimension
$  p=2  $ have been described in \cite{Edel90} and \cite{Souv87}. For other
dimensions, there are probabilistic algorithms producing approximate solutions
of the problem (see \cite{Atki91} and \cite{Joss90}).

The purpose of this paper is to present some new ideas concerning the LMS
problem so as to provide some theoretical framework for efficient regression
algorithms. In Section~2 we offer a useful representation of the problem. The
representation is exploited in Section~3 to demonstrate properties of the
objective function and estimate the number of its local minima. Section~4
includes our main result providing the exact number of local minima. Finally,
in Section~5 we briefly outline three LMS regression algorithms based on the
above results.

\section{Representation of the LMS Problem}

To produce our representations, we first replace (\ref{equ1}) by an equivalent
problem just examined below. Obviously, the solutions of (\ref{equ1}) are
exactly the same as those of the problem:
\begin{equation} \label{equ2}
\min_{\theta} \; \mbox{med} \: \{ |y_{i} - {x}_{i}^{\top} {\theta}| \}.
\end{equation}

A serious difficulty one meets in analyzing both problems (\ref{equ1}) and
(\ref{equ2}) is that it is hard to understand how the median behaves as the
objective function. The next result offers a useful representation for the
median as well as for other operators defined by means of ordering.

Let $  R = \{ r_{1}, \ldots ,r_{n} \}  $ be a finite set of real numbers.
Suppose that we arrange its elements in order of increase, and denote the
$  k $th smallest element by $  r_{(k)} $. If there are elements of the
equal value, we count them repeatedly in an arbitrary order.

\begin{lemma}
For each $  k = 1, \ldots ,n, $ the value of $  r_{(k)}  $ is given by
\begin{equation} \label{equ3}
r_{(k)} = \min_{I \in \Im_{\scriptstyle k}} \; \max_{ i \in I } \: r_{i},
\end{equation}
where $  \Im_{k}  $ is the set of all $k$-subsets of the set
$  N = \{1, \ldots ,n \} $.
\end{lemma}

\begin{proof} Denote the set of indices of the first $  k  $ smallest
elements by $  I^{*} $. It is clear that
$  r_{(k)} = \max_{ i \in I^{*} } r_{i} $. Consider an arbitrary subset
$  I \in \Im_{k} $. Obviously, if $  I \neq I^{*} $, there is at least one
index $  j \in I  $ such that $  r_{j} \geq r_{(k)} $. Therefore, we have
$  r_{(k)} \leq \max_{ i \in I }   r_{i} $. It remains to take minimum over
all $  I \in \Im_{k}  $ in the last inequality so as to get (\ref{equ3}).
\end{proof}

Let $  h= \lfloor n/2 \rfloor + 1 $, where $  \lfloor n/2 \rfloor  $ is the
largest integer less than or equal to $  n/2 $. For simplicity, we assume
$  \mbox{med}_{ i \in N } \: r_{i} = r_{(h)}  $. (It is absolutely correct
to define the median in this form if $  n  $ is odd. However, for an even
$  n $, it is normally defined as $  \frac{1}{2}( r_{(h-1)} + r_{(h)} ) $.)
By using (\ref{equ3}) with $  k = h  $ and
$  r_{i} = r_{i} ( {\theta} ) = |y_{i} - {x}_{i}^{\top} {\theta}| $, we may
now rewrite (\ref{equ2}) as follows:
\begin{equation}\label{equ4}
\min_{\theta} \: \min_{ I \in \Im_{\scriptstyle {h} } } \:
\max_{ i \in I } \: |y_{i} - {x}_{i}^{\top} {\theta}|.
\end{equation}

The obtained representation seems to be more useful than the original because
it is based on the well-known functions {\it max\/} and {\it min}. Moreover,
the representation allows of further reducing the problem. In particular, one
may change the order of the operations of taking minimum in (\ref{equ4}) and
get
\begin{equation} \label{equ5}
\min_{ I \in \Im_{\scriptstyle {h} } } \:
\min_{\theta} \: \max_{ i \in I } \: |y_{i} - {}{x}_{i}^{\top} {}{\theta}|.
\end{equation}

Assume $  I  $ to be a fixed subset of $  N $. Consider the problem
\begin{equation} \label{equ6}
P(I): \;
\min_{\theta} \: \max_{ i \in I } \: |y_{i} - {x}_{i}^{\top} {\theta}|.
\end{equation}
This is the well-known problem of fitting a linear function according to the
$  l_{\infty} $--criterion, first examined by Fourier in the early 19th
century \cite{Four26}. The method proposed by Fourier was actually a version
of the simplex algorithm and therefore (\ref{equ6}) may be regarded as one of
the oldest problems in linear programming. For modern methods and ideas, one
can be referred to \cite{Poly89}. Incidentally, by applying an additional
variable $  \rho $, we may shape (\ref{equ6}) into a usual form of linear
programming problems:
\begin{eqnarray}
& \min \: \rho &                                                  \nonumber \\
& \mbox{subject to} & \rho - x_{i}^{\top}\theta \geq -y_{i},   \; \;
        \rho + x_{i}^{\top}\theta \geq y_{i}, \; i \in I. \label{equ7}
\end{eqnarray}

To conclude this section, note that (\ref{equ5} may be regarded as a
"two--stage" problem of both combinatorial optimization and linear
programming. It consists in minimizing a function defined on a discrete set by
solving some linear programming problem.

\section{An Analysis of the Objective Function}

In this section we examine properties of the objective function in
(\ref{equ4}), {\it i.e.}
\begin{equation}\label{equ8}
F(\theta)=\min_{I \in \Im_{\scriptstyle h}} \: \max_{ i \in I } \:
|y_{i} - {x}_{i}^{\top} {\theta}|.
\end{equation}
The main question we will try to answer is how many local minima it can have.
To start the discussion, consider the function
$  \varrho_{I} (\theta)
= \max_{i \in I}  | y_{i} - x_{i}^{\top} \theta |, \; I \subset N $. It is a
piecewise linear and convex function bounded below. Clearly, the problem of
minimizing $  \varrho_{I}(\theta)  $ always has the solution.

The function $  \varrho_{I}(\theta)  $ can be portrayed as the surface of
a convex polyhedron in a $  (p+1) $--dimensional space. It is not difficult
to see that function (\ref{equ8}), which one may now express as
$ \: F(\theta)= \min_{ I \in \Im_{\scriptstyle h} } \varrho_{I} (\theta) $,
also allows of visualizing its graph as the surface of some polyhedron. It is
that produced by taking union of the polyhedra associated with
$  \varrho_{I}(\theta) $, for all $  I \in \Im_{h} $. Note that
$  F(\theta)  $ is still piecewise linear, but fails to be convex. An
illustration for $  p = 1  $ and $  n = 5  $ is given in Figure~\ref{fig1}.
\begin{figure}[hhh]
\setlength{\unitlength}{2.5mm}
\begin{center}

\begin{picture}(50,33)(0,-5)

\put(7,0){\line(-2,3){7}}
\put(7,0){\line(2,3){17}}
\put(14,0){\line(-1,5){5}}
\put(14,0){\line(1,5){5.2}}
\put(22,0){\line(-1,1){20}}
\put(22,0){\line(1,1){23}}
\put(28,0){\line(-1,2){13}}
\put(28,0){\line(1,2){12}}
\put(42,0){\line(-1,1){26}}
\put(42,0){\line(1,1){6}}
\put(0,0){\vector(1,0){48}}
\multiput(13,9)(0,-1.05){9}{\line(0,-1){0.55}}

\multiput(12,10)(-0.2,1.0){15}{\line(0,1){3}}
\multiput(12,10)(0.25,-0.25){4}{\line(0,1){2}}
\multiput(13,9)(0.25,0.375){16}{\line(0,1){2}}
\multiput(17,15)(0.2,1.0){5}{\line(0,1){3}}
\multiput(18,20)(0.25,-0.5){4}{\line(0,1){2.5}}
\multiput(19,18)(0.25,0.375){8}{\line(0,1){2}}
\multiput(21,21)(0.25,-0.25){44}{\line(0,1){2}}
\multiput(32,10)(0.25,0.25){8}{\line(0,1){2}}
\multiput(34,12)(0.25,0.5){24}{\line(0,1){2.5}}

\thicklines
\put(12,10){\line(-1,5){3}}
\put(13,9){\line(2,3){4}}
\put(12,10){\line(1,-1){1}}
\put(17,15){\line(1,5){1}}
\put(18,20){\line(1,-2){1}}
\put(19,18){\line(2,3){2}}
\put(21,21){\line(1,-1){11}}
\put(32,10){\line(1,1){2}}
\put(34,12){\line(1,2){6}}

\put(6,-2){$y_{1}$}
\put(12,-2){$\theta^{*}$}
\put(14,-2){$y_{2}$}
\put(21,-2){$y_{3}$}
\put(27,-2){$y_{4}$}
\put(41,-2){$y_{5}$}
\put(49,0){$\theta$}
\end{picture}
\caption{An objective function plot.}\label{fig1}
\end{center}
\end{figure}
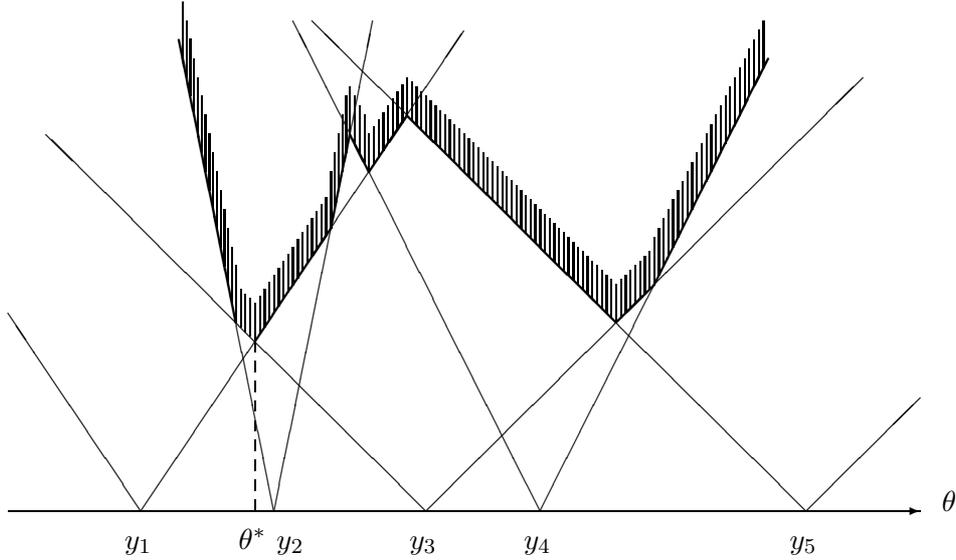

The objective function in Figure~\ref{fig1} is multi-extremal, it has three
local minima. It is clear that for practical problems, the number of the local
minima of (\ref{equ8}) can be enormous. To take the first step to determining
this number, we may conclude from representation (\ref{equ5}) that it must not
be greater than the number of problems $  P(I)  $ for all
$  I \in \Im_{h} $. This last is equal to $  n \choose{h} $, {\it i.e.\/}
the number of all $h$--subsets $  I \in \Im_{h} $.

Suppose $  \theta^{*}  $ to be the solution of a problem $  P(I)  $,
$  |I| \geq p+1  $. One can state the condition for the function
$  \varrho_{I}(\theta)  $ to have the minimum at $  \theta^{*}  $
(see \cite{Poly89}): it is necessary that there exist a $ (p+1)$--subset
$  I^{*} \subset I  $ and real numbers $  \lambda_{i}  $ to satisfy
\begin{equation}\label{equ9}
\sum_{ i \in I^{*} } \: \lambda_{i}   \varepsilon_{i}   x_{i}=0, \; \;
\sum_{ i \in I^{*} } \: \lambda_{i} = 1, \; \; \lambda_{i} \geq 0, \;
i \in I^{*},
\end{equation}
for some $  \varepsilon_{i} \in \{-1,1\} $. In other words,
$  \theta^{*}  $ is defined by the point of intersection of $  p+1  $
"active" hyperplanes
$\: \rho + \varepsilon_{i}   x_{i}^{\top}   \theta
= \varepsilon_{i}   y_{i} \:$ for some choice of $  \varepsilon_{i} \in
\{-1,1\}, \: i \in I^{*} $, provided that the intersection point is an "acute"
top of the corresponding polyhedron.

On the other hand, for any $(p+1)$--subset of indices, we are always able to
choose both $  \lambda_{i}  $ and $  \varepsilon_{i}  $ suitable to
satisfy (\ref{equ9}). To illustrate this, let us examine an arbitrary
$(p+1)$--subset. Without loss of generality, we assume it to be
$  \{1, \ldots , p+1 \} $. Consider the equation
$\; \sum_{i=1}^{p}   t_{i}   x_{i} = -t_{p+1}   x_{p+1}, $ and set
$  t_{p+1} = 1  $ in it. Since $\: \mbox{rank}(X) = p $, we may obtain
values of $  t_{1}, \ldots ,t_{p}  $ as the unique solution of the above
equation. For every
$  i=1, \ldots ,p+1 $, we define $ \; \lambda_{i}
= |t_{i}| / \sum_{j=1}^{p+1}|t_{j}|, \; \; \varepsilon_{i} = \mbox{sign}
(t_{i}) $. Obviously, $  \lambda_{i}, \: i=1, \ldots ,p+1 $, are just those
required in (\ref{equ9}).

As we have shown, the solution of any problem $  P(I)  $ is determined by
$  p+1  $ vectors $  x_{i} $. Conversely, any $  p+1  $ vectors
$  x_{i}  $ produce only one point which satisfies the necessary condition
(\ref{equ9}) and can therefore be treated as the solution of some problem.
Clearly, the number of the local minima of $  F(\theta)  $ must not be
greater than the number of such points, equaled $  {n \choose{p+1}} $. Since
we assume that $  p \leq n/2 $, our first estimate $  {n\choose{h}}  $ can
be improved by replacing it by $  {n \choose{p+1}} $. Although the last
estimate is still rough, yet it is much lower than the quantity $  n^{p}  $
considered in \cite{Edel90,Joss90,Souv87} as the order of the number of local
minima.

\section{The Exact Number of Local Minima}

We may now present our main result providing us with the exact number of
local minima in (\ref{equ4}). In fact, it allows of determining the number of
local minima for any function of the absolute residuals
$\; |y_{i} - x_{i}^{\top} \theta|, \; i \in N \:$, defined by using
representation (\ref{equ3}).

For each $  k=0,1, \ldots , n-(p+1) $, let us introduce the function
\begin{equation}\label{equ10}
f_{k}(\theta) = \min_{ I \in \Im_{\scriptstyle {n-k} } } \:
                \max_{ i \in I } \: |y_{i} - {x}_{i}^{\top} {\theta}|,
\end{equation}
and denote the number of its local minima by $  M_{k} $. It should be noted
that we have to set $  k=n-h=\lfloor \frac{n-1}{2} \rfloor  $ in
(\ref{equ10}) to produce the objective function of problem (\ref{equ4}).

\begin{theorem}\label{the}
For each $  k=0,1, \ldots ,n-(p+1) $, it holds
\begin{equation}\label{equ11}
M_{k} = { p+k \choose{p} }.
\end{equation}
\end{theorem}

\begin{proof}[Sketch of the proof]
Let $  \Pi  $ be the set of problems $  P(I)  $
for all $  I \subset N, \; |I| \geq p+1 $. To prove the theorem, we express
$  | \Pi |  $, {\it i.e.\/} the number of all the problems in $  \Pi $, in
two ways. Firstly, it is easy to see that this number may be calculated as the
sum
\begin{equation}\label{equ12}
| \Pi | = {n \choose{0}} + {n \choose{1}} + \ldots + {n \choose{n-(p+1)}}
        = \sum_{j=0}^{n-(p+1)} {n \choose{j}}.
\end{equation}

To produce the second representation, we examine a local minimum of the
function $  f_{k}(\theta)  $ for an arbitrary
$  k, \; 0 \leq k \leq n-(p+1) $. Assume $  \theta^{*}  $ to be the point
of the local minimum. It is clear that $\: \theta^{*} = \theta^{*}(I) \:$ is
the solution of some problem $  P(I)  $, where $  |I| = n-k $. Since
$  \theta^{*}  $ is actually determined by a subset $  I^{*} \subset I $,
which consists of $  p+1  $ "active" indices, it is also the solution of
problems $  P( I \setminus J )  $ for all $  J \subset I \setminus I^{*} $.
The number of the problems having the solution at $  \theta^{*}  $ coincides
with the number of all subsets of $  I \setminus I^{*}  $ including the
empty set $  \emptyset $, and equals $  2^{n-(p+1)-k} $. In that case, the
total number of the problems connected with the local minima of
$  f_{k}(\theta)  $ is $  2^{n-(p+1)-k} M_{k} $.

Now we may express $  | \Pi |  $ in the form:
\begin{equation}\label{equ13}
| \Pi | = 2^{n-(p+1)} M_{0} + 2^{n-(p+1)-1} M_{1} + \ldots + M_{n-(p+1)}
        = \sum_{j=0}^{n-(p+1)} 2^{n-(p+1)-j} M_{j}.
\end{equation}

From (\ref{equ12}) and (\ref{equ13}), we have
\begin{equation}\label{equ14}
\sum_{j=0}^{n-(p+1)} 2^{n-(p+1)-j} M_{j}
                                       = \sum_{j=0}^{n-(p+1)} {n \choose{j}}.
\end{equation}

It is not difficult to understand that for a fixed
$  k, \; 0 \leq k \leq n-(p+1) $, the number $  M_{k}  $ depends on
$  p  $, but does not on $  n $. One can consider $  M_{0}  $ as an
illustration. Because the problem $  P(N)  $ has the unique solution
(see \cite{Poly89}), $  M_{0}  $ is always equal to $  1 $. Also, it holds
$  M_{1} = p+1  $ independently on $  n $. To see this, note that every one
of the local minima of $  f_{1}(\theta)  $ can be produced by relaxing only
one of $  p+1  $ "active" constraints at the minimum point of
$  f_{0}(\theta) $.

Setting $  n=p+1, p+2, p+3, \ldots \; $ in (\ref{equ14}), we may successively
get $\; M_{0} = 1, \; M_{1} = p+1, \; M_{2} = \frac{(p+1)(p+2)}{2}, \ldots $.
It is not difficult to verify that the general solution of (\ref{equ14}) is
represented as (\ref{equ11}).
\end{proof}

Finally, substituting $\; k=\lfloor \frac{n-1}{2} \rfloor \;$ into
(\ref{equ11}), we conclude that the objective function of the LMS problem has
$ {p+\lfloor(n-1)/2 \rfloor \choose{p}}  $ local minima.

\section{Applications}

In this section we briefly outline LMS regression algorithms based on the
above analysis of the problem. Only the main ideas that underlie the
algorithms are presented.

\paragraph*{"Greedy" algorithm.} The algorithm produces an approximate solution and
consists of solving the sequence of problems (\ref{equ6}),
$  P(I_{0}), P(I_{1}), \ldots, P(I_{n-h}) $, where $  I_{0}=N  $ and the
sets $  I_{1}, I_{2}, \ldots, I_{n-h}  $ are defined as follows. Let
$  I_{k}^{*}  $ be the set of $  p+1  $ "active" indices for the solution
of a problem $  P(I_{k}) $. Clearly, for each $  i \in I_{k}^{*} $, the
minimum of the objective function in the problem
$  P(I_{k}\setminus\{i\})  $ is at least no greater than that in
$  P(I_{k}) $. Denote by $  i_{k}^{*}  $ the index that yields the problem
having the lowest solution. Finally, we define
$  I_{k+1} = I_{k} \setminus \{ i_{k}^{*} \} $.

The "greedy" algorithm formally requires solving
$  (n-h) \! \times \! (p+1)+ 1  $ optimization problems. In practice,
however, an efficient procedure of transition between points, which yields the
solutions of the problems, may be designed to avoid solving each of them.

\paragraph*{Exhaustive search algorithm.} This algorithm may be considered as the
complete version of the previous one which actually uses a reduced search
procedure. It exploits the classical depth-first search technique to provide
all local minima of the objective function. From Theorem~\ref{the}, one can
conclude that it requires examining $  {n-h+p+1\choose{p+1}}  $ points to
produce the exact solution. Because of its exponential time complexity, this
search algorithm can hardly be applied to problems of high dimensions. Note,
however, that it normally allows of solving problems with $  p \leq 5  $
within reasonable time.

\paragraph*{Branch and probability bound algorithm.} It is a random search algorithm
based on the Branch and Probability Bound (BPB) technique which has been
developed in \cite{Zhig91} as an efficient tool for solving both continuous
and discrete optimization problems. The BPB algorithm designed to solve the
LMS problem is of combinatorial optimization. It produces an approximate
solution by searching over $(p+1)$--subsets of $  N $. As it follows from
Section~3, each $(p+1)$--subset determines a point satisfying the condition
(\ref{equ9}), one of such points is the solution of the LMS problem.

In conclusion, I would like to thank Professor A.A.~Zhigljavsky for drawing my
attention to the problem and for valuable discussions, and Professor
A.C.~Atkinson for his kind interest in this work as well as for providing me
with a reprint of paper \cite{Atki91}.

\bibliography{An_analysis_of_the_Least_Median_of_Squares_regression_problem}

\providecommand{\href}[2]{#2}\begingroup\raggedright\begin{thebibliography}{1}

\bibitem{Rous84}
P.~J. Rousseeuw, ``Least median of squares regression,''
  \href{http://dx.doi.org/10.1080/01621459.1984.10477105}{{\em J. Amer.
  Statist. Assoc.} {\bfseries 79} no.~388, (1984) 871--880}.

\bibitem{Edel90}
H.~Edelsbrunner and D.~L. Souvaine, ``Computing least median of squares
  regression lines and guided topological sweep,''
  \href{http://dx.doi.org/10.1080/01621459.1990.10475313}{{\em J. Amer.
  Statist. Assoc.} {\bfseries 85} no.~409, (1990) 115--119}.

\bibitem{Joss90}
J.~Joss and A.~Marazzi, ``Probabilistic algorithms for least median of squares
  regression,'' \href{http://dx.doi.org/10.1016/0167-9473(90)90075-S}{{\em
  Comput. Statist. Data Anal.} {\bfseries 9} no.~1, (1990) 123--133}.

\bibitem{Souv87}
D.~L. Souvaine and J.~M. Steele, ``Time- and space-efficient algorithms for
  least median of squares regression,''
  \href{http://dx.doi.org/10.1080/01621459.1987.10478500}{{\em J. Amer.
  Statist. Assoc.} {\bfseries 82} no.~399, (1987) 794--801}.

\bibitem{Atki91}
A.~Atkinson and S.~Weisberg, ``Simulated annealing for the detection of
  multiple outliers using least squares and least median of squares fitting,''
  in {\em Directions in Robust Statistics and Diagnostics}, W.~Stahel and
  S.~Weisberg, eds., pp.~7--20.
\newblock Springer, 1991.

\bibitem{Four26}
J.~Fourier, ``Analyse de travaux de l'\uppercase{A}cademie \uppercase{R}oyale
  des \uppercase{S}ciences pendant l'ann\'{e}e 1823, \uppercase{P}artie
  \uppercase{M}ath\'{e}matique,'' {\em Histoire de l'Academie Royale des
  Sciences de l'Institute de France} {\bfseries 6} (1826) xxix--xli.

\bibitem{Poly89}
B.~T. Polyak, {\em Introduction to optimization}.
\newblock Optimization Software, New York, 1987.

\bibitem{Zhig91}
A.~A. Zhigljavsky, {\em Theory of Global Random Search}.
\newblock Kluwer Academic Press, Dordrecht, 1991.

\end{thebibliography}\endgroup

\end{document}